\documentclass[a4paper]{article}
\usepackage{latexsym}
\usepackage{amsmath,amssymb,amsfonts}

\newtheorem{prethm}{{\bf Theorem}}

\newtheorem{prepro}[prethm]{Proposition}
\newenvironment{pro}{\begin{prepro}{\hspace{-0.5
               em}{\bf.}}}{\end{prepro}}

\newtheorem{prelem}[prethm]{Lemma}
\newenvironment{lem}{\begin{prelem}{\hspace{-0.5
               em}{\bf.}}}{\end{prelem}}

\newtheorem{precor}[prethm]{Corollary}

\newtheorem{prerem}[prethm]{{\bf Remark}}
\newenvironment{rem}{\begin{prerem}\em{\hspace{-0.5
              em}{\bf.}}}{\end{prerem}}

\newtheorem{preconj}[prethm]{{\bf Conjecture}}

\newtheorem{preexample}{{\bf Example}}

\newtheorem{preproof}{{\bf Proof.}}

\newenvironment{proof}[1]{\begin{preproof}{\rm
               #1}\hfill{$\Box$}}{\end{preproof}}

\newcommand{\noi}{\noindent}

\newcommand{\st}{{\rm ST}}
\newcommand{\sts}{{\rm STS}}
\newcommand{\be}{\beta}
\newcommand{\al}{\alpha}
\newcommand{\la}{\lambda}

\newcommand{\T}{{\rm T}}
\newcommand{\si}{\sigma}
\newcommand{\C}{{\cal C}}
\newcommand{\B}{{\cal B}}
\newcommand{\F}{{\cal F}}
\newcommand{\x}{{\bf x}}
\newcommand{\y}{{\bf y}}
\newcommand{\bmi}[1]{\mbox{\boldmath $ #1$}}

\title{Simple signed Steiner triple systems}
\author{ E. Ghorbani$^{\textrm{a,b}}$ ~~~ G.B. Khosrovshahi$^{\textrm{b,c,}}$\thanks{Corresponding author.}\\
{\small {$^{\rm a}$Department of Mathematics, K.N. Toosi University of Technology,}}\\
{\small P. O. Box 16315-1618, Tehran, Iran}\\
 {\small {$^{\rm b}$School of Mathematics, Institute
for Research in Fundamental Sciences (IPM),}} \\
 {\small { P.O. Box 19395-5746, Tehran, Iran}}\\
 {\small { $^{\rm c}$Department of Mathematics, Statistics and Computer Science, University of Tehran, Iran}}\\
 {\tt\small e\_ghorbani@ipm.ir~~~rezagbk@ipm.ir}}
\begin{document}
\maketitle

\begin{abstract}

Let $X$ be a  $v$-set,  $\B$  a set of 3-subsets (triples) of $X$, and $\B^+\cup\B^-$ a partition of $\B$ with  $|\B^-|=s$.
   The pair $(X,\B)$ is called a simple signed Steiner triple system, denoted by ST$(v,s)$, if the number of occurrences of every 2-subset of $X$ in triples $B\in\B^+$ is one more than the  number of occurrences in  triples $B\in\B^-$.
In this paper we prove that $\st(v,s)$ exists if and only if $v\equiv1,3\pmod6$, $v\ne7$,   and $s\in\{0,1,\ldots,s_v-6,s_v-4,s_v\}$, where $s_v=v(v-1)(v-3)/12$ and for $v=7$, $s\in\{0,2,3,5,6,8,14\}$.

\vspace{3mm}
\noindent {\em AMS subject Classification}:  05B07\\
\noindent{\em Keywords}: Simple Signed Steiner Triple System, Trade
\end{abstract}

\section{Introduction}

Let $X$ be a finite set and let ${X\choose k}$ denote the set of all $k$-subsets of $X$.
If $\B$ is a subset of the power set of $X$, then for any $\al\subseteq X$, we define
$m(\al,\B)$ to be the number of $B\in\B$ such that $\al\subseteq B$.
A {\em signed set} is a set with an assignment of $+$ or $-$ to its elements.
We say that a subset $\al$ is $t$-{\em balanced} in a signed set $\B$ if $m(\al,\B^+)-m(\al,\B^-)=t$.
Let $Y$ be a signed set.  We denote the set of positive and negative elements of $Y$ by $Y^+$ and $Y^-$, respectively.
We denote the 2-subset $\{x,y\}$  by $xy$.
A {\em simple signed Steiner triple system} is a pair $(X,\B)$, where $\B\subseteq{X\choose3}$ is a signed set such that every pair $xy$ of $X$ is 1-balanced in $\B$.
 If $|X|=v$ and $|\B^-|=s$, we denote the simple signed Steiner triple system $(X,\B)$ by $\st(v,s)$.
Clearly, an  $\st(v,0)$ is a Steiner triple system, $\sts(v)$.

A $t$-$(v,k,\lambda)$ {\em design} is a pair $(X,\B)$  where $X$ is a $v$-set and  $\B$ is a collection of $k$-subsets of $X$ such that
every $t$-subset of $X$ occurs exactly $\la$ times in blocks $B\in\B$.
If $\B$ has no repeated blocks, then the design is called {\em simple}. The conditions
\begin{equation}\label{nec}
   \la_i={v-i\choose t-i}/{k-i\choose t-i}\in\mathbb{Z},~~i=0,1,\ldots,t,
\end{equation}
are necessary for the existence of a $t$-$(v,k,\lambda)$ design.
 The {\em inclusion matrix} $W_{tk}^v$  is a $(0,1)$-matrix whose
rows and columns are indexed  by $t$-subsets and $k$-subsets of $X$, respectively, and $W_{tk}^v(T,K)=1$ if and only if
$T\subseteq K$.
 In terms of inclusion matrices,  $t$-designs can be described in a linear-algebraic language.
 Let $\x$ be the characteristic vector of length ${v\choose k}$ for the block set of a $t$-$(v,k,\lambda)$ design, then $\x$ is a solution for
  \begin{equation}\label{w}
   W_{tk}^v{\bf x}=\lambda{\bf1},
 \end{equation}
 where ${\bf1}$ is the all-1 vector.
 This is a motivation to generalize the concept of $t$-designs to signed $t$-designs: every integral solution of (\ref{w}) is called a {\em signed $t$-$(v,k,\lambda)$  design}.
 We remark that the conditions (\ref{nec}) are necessary and sufficient for the existence of an integral solution for (\ref{w}). This is proved in \cite{gju,w}, see also \cite{gkmm}.

The aim of this paper is to prove the following theorem.

\noi{\bf Theorem.} {\em For $v\ne7$, there exists an $\st(v,s)$ if and only if $v\equiv1,3\pmod6$ and $s\in\{0,1,\ldots,s_v-6,s_v-4,s_v\}$, where $s_v=v(v-1)(v-3)/12$. Moreover, for $v=7$, an $\st(7,s)$ exists if and only if
$s\in\{0,2,3,5,6,8,14\}$.}

\noi Alternatively, in the language of inclusion matrices, the theorem describes the $(-1,0,1)$-solutions $\x$ of $W_{2,3}^v{\bf x}={\bf1}$.
We  remark that the $(0,1)$-solutions  of $W_{2,3}^v{\bf x}={\bf1}$ are Steiner triple systems.

In an attempt to construct random $\sts(v)$, Cameron \cite{c} makes use of $\st(v,1)$ and calls these objects `improper STS'. Our objective in this paper may also be interpreted as a determination of objects with much more `improperness'.

\section{Preliminaries}

Every integral solution of (\ref{w}) with $\la=0$ is  a $\T(t,k,v)$ trade.
Combinatorially, a $\T(t, k, v)$ trade $T$ on $X$ is a nonempty signed collection of blocks from ${X\choose k}$
such that the number of times that each element of ${X\choose t}$ occurs in $T^+$ is the
same as the number of times it occurs in $T^-$. In other words every $t$-subset of $X$ is 0-balanced in $T$.
The value $s=|T^+|=|T^-|$ is called the {\em volume} of the trade $T$. For more on trades see \cite{hk}.
If $X_1$ and $X_2$ are two disjoint sets,
 $\B_1\subseteq{X_1\choose k_1}$ and  $\B_2\subseteq{X_2\choose k_2}$, then we define
 $$\B_1\cdot\B_2=\{B_1\cup B_2:B_1\in\B_1, B_2\in\B_2\}.$$
By a {\em 1-factor} ({\em 1-factorization}) of ${X\choose2}$ we mean a 1-factor (1-factorization) of the complete graph with vertex set $X$ and edge set ${X\choose2}$.

\subsection{Spectrum of simple trades}

We make use of the following result of \cite{bkt}.
\begin{lem} \label{bkn} {\rm(i)} If $v\equiv1\pmod4$ and $v\ge4$, then there exists a simple $\T(1,2,v)$ trade  of volume $s$ if and only if $s\in\{2,3,\ldots,\frac{1}{2}{v\choose2}-2,\frac{1}{2}{v\choose2}\}$.\\
  {\rm(ii)} If $v\equiv2\pmod4$, then there exists a simple $\T(2,3,v)$ trade of volume $s$ if and only if $s\in\{4,6,7,\ldots,\frac{1}{2}{v\choose3}-6,\frac{1}{2}{v\choose3}-4,\frac{1}{2}{v\choose3}\}.$ 
\end{lem}

\begin{rem} If $v\equiv3\pmod4$, then the maximum volume of a $\T(1,2,v)$ trade on a $v$-set $X$ is $\frac{1}{2}{v\choose2}-\frac{3}{2}$ (\cite{bkt}). If $T$ is such a trade on $X$, then it is easily seen that ${X\choose2}\setminus T=\{x_1x_2,x_1x_3,x_2x_3\}$ for some $x_1,x_2,x_3\in X$.
\end{rem}
\begin{lem}\label{seq}  Let $A=\{a_1,\ldots,a_n\}$ and $B=\{b_1, b_1 + 1,\ldots , b_2\}$, where $a_i$, $b_1,b_2$ are integers such that $a_i>a_{i-1}$ and $b_2-b_1\ge a_i -a_{i-1}-1$
for $2\le i\le n$. Then $\{a+b\mid a\in A,b\in B\}$ contains all integers $a_1 + b_1\le c\le a_n + b_2.$
\end{lem}

\begin{lem}\label{remain}
 Let  $v=4k\ge8$ and let  $X_1,\ldots,X_k$ be a partition of a $v$-set $X$ into $4$-subsets.
Then for every  $s\in I_v=\{4,6,7,\ldots,t_v-6,t_v-4,t_v\}$, where $t_v=\frac{1}{2}{v\choose3}-\frac{v}{2}=v(v+1)(v-4)/12$, there exists a $\T(2,3,v)$ trade of volume $s$ such that $T\subseteq{X\choose3}\setminus\bigcup_{i=1}^k{X_i\choose3}$.
\end{lem}
\begin{proof}{We prove the lemma by induction on $v$. For $v=8$, the assertion holds as shown in the appendix. Let $v\ge8$, we prove the lemma for $v+4$. Let $X$ be a set with $v+4$ points and let $X_0,X_1,\ldots,X_k$ be a partition of $X$ into $4$-subsets. Let $X'=X\setminus X_0$. Then,
$${X\choose3}\setminus\bigcup_{i=0}^k{X_i\choose3}={X'\choose3}\cup \bigcup_{i=1}^k\B_i\cup \bigcup_{1\le i<j\le k}X_0\cdot X_i\cdot X_j,$$
where $\B_i=X_0\cdot{X_i\choose2}\cup X_i\cdot{X_0\choose2}$.
 For each pair $ij$,  $1\le i<j\le k$, all the triples of  $X_0\cdot X_i\cdot X_j$ are covered by a trade of volume 32.
From the union of these trades,  a $\T(2,3,v)$ trade $T_1$ of volume $32{k\choose2}=v(v-4)$ is obtained. Since the assertion is true for $v=8$, the triples of each $\B_i$ are covered by a trade of volume 24.
The union of these trades gives a trade $T_2$ of volume $24k=6v$. By the induction hypothesis, for every $s\in I_v$, there exists a $\T(2,3,v)$ trade on $X'$ of volume $s$.
Now, by Lemma~\ref{seq}, for every $s\in I_{v+4}$ we can construct a trade of volume $s$ through the union $T\cup T_1\cup T_2$, for some $\T(2,3,v)$ trade $T$ on $X$.
}\end{proof}

\subsection{Necessary conditions}

In this part we describe the necessary conditions for the existence of an $\st(v,s)$, namely the three conditions: $v\equiv1,3\pmod6$, $s\le s_v$, and $s\not\in\{s_v-5,s_v-3,s_v-2,s_v-1\}$.
Some remarks regarding these conditions are in order.
\begin{itemize}
  \item The necessary conditions for the existence of a $t$-$(v,k,\lambda)$ design are also necessary for the existence of a signed $t$-$(v,k,\lambda)$  design, see \cite{gju,w}. Therefore, in the case of
Steiner triple systems,  the necessary condition $v\equiv1,3\pmod6$ is also necessary for the existence of an $\st(v,s)$.

  \item If $(X,\B)$ is an $\st(v,s)$, then $|\B^+|=s+v(v-1)/6$. So $|\B|=2s+v(v-1)/6\le{v\choose3}$. This implies that $s\le s_v= v(v-1)(v-3)/12$.

  \item Let $(X,\B)$ be an $\st(v,s)$ and let $\overline\B={X\choose3}\setminus\B$.
     For any $x,y\in X$, we have
   $$m(x,\overline\B)={v-1\choose2}-\frac{v-1}{2}-2m(x,\B^-),~~m(xy,\overline\B)=(v-2)-2m(xy,\B^-)-1,$$
both of which are even numbers.  With the above properties, the proof of non-existence of  $\st(v,s)$ for $s\in\{s_v-5,s_v-3,s_v-2,s_v-1\}$ is similar to the proof of non-existence of a $\T(2,3,n)$ trade of volume $s\in \{t_n-5,t_n-3,t_n-2,t_n-1\}$, where $t_n= \frac{1}{2}{n\choose3}$ and $n\equiv2\pmod4$, for this see \cite{bkt}.
\end{itemize}

\section{ Construction of $\bmi{\st(v,s)}$ for $\bmi{v\equiv3,7\pmod{12}}$}

Let $n\equiv1,3\pmod{6}$, $n\ge7$, we prove the theorem for $v=2n+1$. This proves the theorem for $v\equiv3,7\pmod{12}$, $v>7$. We do this in the following two subsections.
Let $X=\{x_1,\ldots,x_n\}$ and $Y=\{y_1,\ldots,y_{n+1}\}$. Let $(X,\B)$ be an $\sts(n)$. Suppose that $F_1,\ldots,F_n$ is a 1-factorization of ${Y\choose2}$ and
  \begin{equation}\label{b'}
  \B':=\bigcup_{i=1}^nx_i\cdot F_i.
\end{equation}
Then $(X\cup Y,\B\cup\B')$ is an $\sts(v)$.

\subsection{$\bmi{0\le s\le5}$}
 We show that there exists an $\st(v,s)$ for $1\le s\le5$.

 Without loss of generality assume that  $y_iy_7\in F_i$, for $i=1,\ldots,6$.
Suppose that $y_1y_4\in F_{j_1}$, $y_2y_5\in F_{j_2}$, and $y_3y_6\in F_{j_3}$ for some $j_1,j_2,j_3$. One can find a permutation $\sigma$ on $X$ with the property that
none of the $x_1^\si x_4^\si x_{j_1}^\si$, $x_2^\si x_5^\si x_{j_2}^\si$, $x_3^\si x_6^\si x_{j_3}^\si$ are triples of $\B$.
Hence we can rearrange $x_i$ so that they satisfy the above property. 
We define $$P(x,y,z;\,a,b,c)=+\{xyz,xbc,yac,zab\}\cup -\{abz, acy,bcx,abc\},$$
and set $P_k=P(y_k,y_{k+3},y_7;\,x_{k+3},x_k,x_{j_k}),$ for $k=1,2,3$. Now for $s=1,2,3$, let $\B_s=\B\cup\B'\cup\bigcup_{k=1}^sP_k$. Then the triple system $(X\cup Y,\,\B_s)$ gives an $\st(v,s)$.
 To see this, for example for $s=1$, we have $$P_1=\{y_1y_4y_7,y_1x_1x_{j_1},y_4x_3x_4,y_7x_1x_4\}\cup-\{x_1x_4x_{j_1},x_1y_1y_7,x_{j_1}y_1y_4,x_4y_4y_7\}.$$
 Then $(\B\cup\B')\cap P_1^+=\emptyset$ and all the triples of $P_1^-$ belong to $\B\cup\B'$ except for $-x_1x_4x_{j_1}$. So $\B\cup\B'\cup P_1$ is an $\st(v,1)$.

 Further if $\B_4=\B\cup\B'\cup P(y_1,y_2,y_3;\,y_4,y_5,y_6)$ and $\B_5=\B_1\cup P(y_1,y_2,y_3;\,y_4,y_5,y_6)$, then $(X\cup Y,\,\B_4)$ and $(X\cup Y,\,\B_5)$ are  $\st(v,4)$ and $\st(v,5)$, respectively.

\subsection{$\bmi{s\ge6}$}
For any $n$-set $X$ with $n\equiv1,3\pmod6$ and $n>7$, it is well known that ${X\choose3}$ possesses a large set of Steiner triple systems (see \cite{cr}), i.e. there is a partition $\C_1\cup\cdots\cup\C_{n-2}$ of ${X\choose3}$ such that  all $(X,\C_i)$, $i=1,\ldots,n-2$, are $\sts(n)$. Thus
\begin{equation}\label{large}
T_1=\bigcup_{i=2}^{n-2}(-1)^i\C_i
\end{equation}
 is a trade of volume $s_1=n(n-1)(n-3)/12$.
Let $\B=\C_1\cup\B'$ with $\B'$ as in (\ref{b'}). Let $T$ be a $\T(2,3,v)$ trade of volume $s$ on $X\cup Y$. Then  $(X\cup Y,\,\B\cup T)$ is an $\st(v,s)$.
 So in order to prove the theorem we show that for every $s\in\{6,\ldots,s_v-6,s_v-4,s_v\}$ there exist a trade $T$ of volume $s$ with triples in ${X\cup Y\choose3}\setminus\B$.

 We consider two cases: $n=4k-1$ and $n=4k+1$.

\noi{\em Case 1. $n=4k-1$}

 We have $${X\cup Y\choose3}={X\choose3}\cup \left(X\cdot{Y\choose2}\right)\cup\left(Y\cdot{X\choose2}\right)\cup{Y\choose3}.$$
Let $Y_1,\ldots,Y_k$ be a partition of $Y$ into $4$-subsets and let ${Y_i\choose3}=\{\al_{i,1},\al_{i,2},\al_{i,3},\al_{i,4}\}$. We cover all the triples of ${Y\choose3}$ as follows:
\begin{itemize}
  \item[-]  $T_2$,  any $\T(2,3,n+1)$ trade on $Y$ of volume $s_2\in \{6,7,\ldots,t_{n+1}-6,t_{n+1}-4,t_{n+1}\}$, where $t_{n+1}=(n+1)(n+2)(n-3)/12$ which exists by Lemma~\ref{remain};
  \item[-] $\B_1=\bigcup_{i=1}^k\{\al_{i,1},\al_{i,2},-\al_{i,3},-\al_{i,4}\}$ which is of volume $(n+1)/2$.
\end{itemize}
Then every 2-subset of $Y$ is 0-balanced in $\B_1$ except for the elements of a 1-factor of ${Y\choose2}$ such that half of its elements are $2$-balanced and the elements of the other half are $-2$-balanced. Suppose this 1-factor consists of $E_2$ and $E_3$ with
$$E_2=\{y_1y_3,y_5y_7,\ldots,y_{n-2}y_n\}~~\hbox{and}~~E_3=\{y_2y_4,y_6y_8,\ldots,y_{n-1}y_{n+1}\},$$
and the elements of $E_2$ and $E_3$ are $-2$- and $2$-balanced, respectively.

Now we deal with $X\cdot{Y\choose2}$. Let $L$ be a Latin square of order $n$ with entries in $X$ such that its first row is $(x_1,\ldots,x_n)$.
Consider the signed set
\begin{equation}\label{latin}
\bigcup_{i=2}^n\bigcup_{j=1}^n(-1)^iL(i,j)\cdot F_j.
\end{equation}
Note that $E_2$ and $E_3$ could be extended to two disjoint 1-factors of ${Y\choose2}$, say $F_2$ and $F_3$. So we may assume that $F_2=E_2\cup E_2'$, $F_3=E_3\cup E_3'$ for some $E_2',E_3'$.
(We remark that this is possible as any $(n-3)$-regular graph of order $n$ with $n$ even possesses a 1-factorization.)
We may assume that in the above union $x_1\cdot F_2$ and $x_1\cdot F_3$ appear with negative and positive signs, respectively.
Now change the signs of triples of these two sets to
$$x_1\cdot(E_2 \cup - E_2'\cup - E_3 \cup  E_3'),$$
and call the resulting singed set $\B_2$ which is of volume $\frac{n-1}{2}{n+1\choose2}$. Then every 2-subset of $X\cup Y$ is 0-balanced except for $\al\in E_2$ and $\be\in E_3$ with $m(\al,\B_2)=2$ and  $m(\be,\B_2)=-2$; and $m(x_1y_i,\B_2)=(-1)^{i+1}2$ for $i=1,\ldots,n+1$.

 Let $\{x_1x_2,x_1x_3,x_2x_3\}$ be the remaining blocks not covered by a $\T(1,2,n)$ trade of the maximum volume $(n+2)(n-3)/4$ on $X$.  We cover $Y\cdot{X\choose2}$ by
\begin{itemize}
  \item[-] the trade $T_3=\bigcup_{i=1}^{2k}\{-y_{2i-1},y_{2i}\}\cdot T_3'$, where $T_3'$ is a $\T(1,2,n)$ trade on $X$ of volume either $(n^2-1)/8$ or
    $(n+2)(n-3)/4$. So $T_3$ is of volume $s_3\in\{(n+1)(n^2-1)/8,(n+1)(n+2)(n-3)/4\}$.
  \item[-] $\B_3=\{x_1x_2,x_1x_3,-x_2x_3\}\cdot\bigcup_{i=1}^{2k}\{-y_{2i-1},y_{2i}\}$ which is of volume $3(n+1)/2$.
\end{itemize}
Then every 2-subset of $X\cup Y$ is 0-balanced in $\B_3$ except for $m(x_1y_i,\B_3)=(-1)^i2$ for $i=1,\ldots,n+1$.
Hence $T_4=\B_1\cup\B_2\cup\B_3$ is a trade of volume $s_4=2(n+1)+\frac{n-1}{2}{n+1\choose2}$.
Therefore, from Lemma~\ref{remain} it follows that by taking appropriate unions of trades $T_1,T_2,T_3,T_4$ we are able to construct a $\T(2,3,v)$ trade on $X\cup Y$ with volume $s$ for every
$s\in\{6,\ldots,s_v-6,s_v-4,s_v\}$. This completes the proof in this case.

\noi{\em Note:} For $v=15$ and correspondingly $n=7$, the trade $T_1$ as in (\ref{large}) does not exist since the associated large set does not by a result due to Cayley (see \cite{cr}). So if we let $T_1=\emptyset$
in all the above arguments, we obtain $\st(15,s)$ for $s\in\{6,\ldots,\ell-6,\ell-4,\ell\}$ where $\ell=s_{15}-14$. Now, let $\C_0=\F_1\cup\F_2\cup-\F_3$ where
$\F_1$ and $\F_2$ are the block sets of two disjoint Fano planes and $\F_3={X\choose3}\setminus(\F_1\cup\F_2)$ is the block set of a 2-$(7,3,3)$ design. Then $(X\cup Y, \C_0\cup\B')$
is an $\st(15,14)$. Now if we use $\C_0$ instead of $\C_1\cup T_1$ in  the above arguments, we obtain $\st(15,s)$ for the remaining values of $s$.


\noi{\em Case 2. $n=4k+1$}

Let $T_2$ be any $\T(2,3,n+1)$ trade of volume $s\in\{6,\ldots,\frac{1}{2}{n+1\choose3}\}$ on $Y$ which exists by Lemma~\ref{bkn}. Let $T_3$ be as in (\ref{latin}) and $T_4$ be the trade $T_4=\bigcup_{i=1}^{2k}\{-y_{2i-1},y_{2i}\}\cdot T_4'$, where $T_4'$ is a $\T(1,2,n)$ trade on $X$ of volume either $(n^2-1)/8$ or
      $\frac{1}{2}{n\choose2}$. So $T_4$ is of volume $s_3\in\{(n+1)(n^2-1)/8,n(n+1)(n-1)/4\}$.
Then again we see that by taking appropriate unions of trades $T_1,T_2,T_3,T_4$, we are able to construct a $\T(2,3,v)$ trade on $X\cup Y$ with volume $s$ for every
$s\in\{6,\ldots,s_v-6,s_v-4,s_v\}$.


\section{Construction of $\bmi{\st(v,s)}$ for $\bmi{v\equiv1,9\pmod{12}}$}
Let $n\equiv1,3\pmod{6}$, we prove the theorem for $v=2n+7$. This proves the theorem for $v\equiv1,9\pmod{12}$.

\subsection{$\bmi{0\le s\le 6n}$}

Let $Y=\{y_1,\ldots,y_{n+7}\}$ and $(X,\B_1)$ be an $\sts(n)$.
Let $\B_2=\{y_iy_{i+1}y_{i+3}\mid i\in\mathbb{Z}_{n+7}\}$.
 If we consider the triples of $\B_2$ as triangles in the complete graph with vertex set $Y$, and remove the edges of these triangles,  the remaining graph  is $n$-regular. This graph has a 1-factorization $E_1,\ldots,E_n$. Then
with $\B=\B_1\cup\B_2\cup\bigcup_{i=1}^nx_i\cdot E_i$ the triple system $(X\cup Y,\,\B )$ is an $\sts(v)$ (see \cite{cr}).
We notice that for $k=1,2,3$, $y_ky_{k+3}y_7$ does not belong to $\B_2$. Hence the same construction as the one applied in 3.1 gives rise to an $\st(v,s)$ for $s=1,2,3$.
Now for $4\le s\le6n$ and $n\ge7$ consider the following construction. (The cases $n=1,3$ will be treated at the end of this section). Let $T_2'$ be any $\T(1,2,n)$ trade on $X$ with volume $s\in\{2,3,\ldots,n\}$. Suppose $T_2$ is one of the trades $\{y_8,-y_9\}\cdot T_2'$ or $\bigcup_{i=8}^{10}\{-y_i,y_{i+3}\}\cdot T_2'$. Then $T_2$ is a $\T(2,3,v)$ trade with triples in $Y\cdot{X\choose2}$  of volume $s$ for every $s\in\{4,6,8,\ldots,6n\}$. Then $\B\cup T_2$ gives an $\st(v,s)$ for  $s\in\{4,6,8,\ldots,6n\}$.

\subsection{$\bmi{6n\le s\le s_v}$}

\noi{\em Case 1. $n=4k+1$}

For $n=4k+1$, we have $|Y|=4(k+2)$.
Let $Y_1,\ldots,Y_{k+2}$ be a partition of $Y$ into $4$-subsets. Then $\bigcup_{i=1}^{k+2}{Y_i\choose2}$ can be thought of as a union of three 1-factors of ${Y\choose2}$, say $F_{n+4}\cup F_{n+5}\cup F_{n+6}$. Removing these 1-factors from ${Y\choose2}$, the remaining graph is regular on an even number of vertices and with the even valency $n+3$. Thus it has a 1-factorization $F_1,\ldots,F_{n+3}$.
Let ${Y_i\choose3}=\{\al_{i,1},\al_{i,2},\al_{i,3},\al_{i,4}\}$. We cover the triples of ${Y\choose3}$ through the following:
\begin{itemize}
  \item[-]  $T_1$,  any $\T(2,3,n+7)$ trade on  $Y$ of volume $s_2\in \{6,7,\ldots,t_{n+7}-6,t_{n+7}-4,t_{n+7}\}$;
  \item[-] $\B_1=\bigcup_{i=1}^{k+2}\{\al_{i,1},\al_{i,2},\al_{i,3},\al_{i,4}\}$.
\end{itemize}

We cover $X\cdot{Y\choose2}$ by $\B_2\cup T_3$, where
\begin{itemize}
 \item[-] $\B_2=\bigcup_{i=1}^nx_i\cdot F_i  \cup x_1\cdot T_3';$
  \item[-] $T_2=\bigcup_{i=2}^n\bigcup_{j=1}^n(-1)^iL(i,j)\cdot F_j\cup\{x_2,\ldots,x_{2k+1},-x_{2k+2},\ldots,-x_n\}\cdot T_2';$
 \end{itemize}
where $$T_2'=F_{n+1}\cup F_{n+2}\cup F_{n+3}\cup-F_{n+4}\cup-F_{n+5}\cup-F_{n+6}.$$

In $Y\cdot{X\choose2}$ consider $T_3=\bigcup_{i=1}^{2k+4}\{-y_{2i-1},y_{2i}\}\cdot T_3'$, where $T_3'$ is a $\T(1,2,n)$ trade on $X$ of volume either $(n^2-1)/8$ or
    $\frac{1}{2}{n\choose2}$. So $T_4$ is a $\T(2,3,v)$ trade on $X\cup Y$ of volume $s_4\in\{(n+7)(n^2-1)/8,\,n(n+7)(n-1)/4\}$.

In $\B_1$, every 2-subset of $Y$ is 0-balanced  except for the elements of $F_{n+4}\cup F_{n+5}\cup F_{n+6}$
which are $2$-balanced. In $\B_2$ every 2-subset of $Y$ is 1-balanced except for the elements of $F_{n+4}\cup F_{n+5}\cup F_{n+6}$ which are $-1$-balanced. Also every element of $X\cdot Y$
is 1-balanced in $\B_2$. Thus in $\B_1\cup\B_2$, every 2-subset of $Y$ and every element of $X\cdot Y$ are 1-balanced.
Let $\C_i$ be as in 3.2 and $T_5$ as in (\ref{large}).
Therefore, with $\B=\C_1\cup\B_1\cup\B_2$, $(X\cup Y,\B)$ is an $\st(v,s)$ with $s=3(n+7)/2$. If we let $T$ be an appropriate union of trades $T_1,\ldots,T_5$, then by $(X\cup Y,\B\cup T)$, an $\st(v,s)$
can be obtained for $s\in\{3(n+7)/2+6,\ldots, s_v-6,s_v-4,s_v\}$.


\noi{\em Case 2. $n=4k-1$}

Assume that $U=X\cup\{y_{n+7}\}$ and $Z=Y\setminus\{y_{n+7}\}$. Then $|U|=n+1=4k$ and $|Z|=n+6=4(k+1)+1$.
Let $U_1,\ldots,U_k$ be a partition of $U$ into $4$-subsets and ${U_i\choose3}=\{\al_{i,1},\al_{i,2},\al_{i,3},\al_{i,4}\}$.
In a similar manner as in Case~1, suppose that $F_1,\ldots,F_n$ is a 1-factorization of ${U\choose2}$ where $F_{n-2}\cup F_{n-1}\cup F_n$ is the edge set of  $\bigcup_{i=1}^k{U_i\choose2}$.
 We cover the triples of ${U\choose3}$ through the following:
\begin{itemize}
  \item[-]  $T_2$,  any $\T(2,3,n+1)$ trade on  $U$ of volume $s_2\in \{6,7,\ldots,t_{n+1}-6,t_{n+1}-4,t_{n+1}\}$;
  \item[-] $\B_1=-\bigcup_{i=1}^k\{\al_{i,1},\al_{i,2},\al_{i,3},\al_{i,4}\}$.
\end{itemize}
Let $L$ be a Latin square of order $n$ on $\{y_7,\ldots,y_{n+7}\}$ such that its first row is $(y_7,\ldots,y_{n+7})$. We cover $Z\cdot{U\choose2}$ by:
\begin{itemize}
 \item[-] $\B_2=\bigcup_{i=1}^ny_{i+6}\cdot F_i;$
  \item[-] $T_3=\bigcup_{i=2}^n\bigcup_{j=1}^n(-1)^iL(i,j)\cdot F_j$;
  \item[-] $\B_3=\{y_1,y_2\}\cdot F_1\cup\{y_3,y_4\}\cdot F_2\cup\{y_5,y_6\}\cdot F_3$;
   \item[-] $T_4=\{y_1,y_2\}\cdot T'_1\cup\{y_3,y_4\}\cdot T'_2\cup\{y_5,y_6\}\cdot T'_3$;
 \end{itemize}
  where $T_i$ is the trade obtained from $\{F_1,\ldots,F_n\}\setminus F_i$, $i=1,2,3$, by negating half of the $F_i$.
   In $U\cdot{Z\choose2}$, let  $T_5$ be the trade
   $$T_5=\left(\{-x_n,y_{n+7}\}\cup\bigcup_{i=1}^{2k-1}\{-x_{2i-1},x_{2i}\}\right)\cdot T_5',$$
    where $T_5'$ is a $\T(1,2,n+6)$ trade on $Z$ of volume either $(n^2-1)/8$ or
      $\frac{1}{2}{n+6\choose2}$. So $T_5$ is of volume $s_5\in\{(n+1)(n^2-1)/8,(n+1)(n+6)(n+5)/4\}$.
Let $\C_1\cup\cdots\cup\C_{n+4}$ be a large set of ${Z\choose3}$ and
$T_6=\bigcup_{i=2}^{n+4}(-1)^i\C_i$. With $\B=\C_1\cup\B_1\cup\B_2\cup\B_3$, $(X\cup Y,\B)$ is an $\st(v,n)$.
By taking appropriate unions of trades $T_1,\ldots,T_6$, we are able to construct a $\T(2,3,v)$ trade on $X\cup Y$ with volume $s$ for every
$s\in\{n,\ldots,s_v-6,s_v-4,s_v\}$.

\noi{\em Note:} For $n=3$ and $v=13$, $\st(v,s)$ for $s=1,2,3$ are constructed in 4.1 and for $s\ge4$ in Case~2 of 4.2.
For $n=1$ and $v=9$, the only remaining volumes are $4\le s\le12$. But this can be easily done by taking appropriate union of $\st(9,s)$ for $s=1,2,3$ obtained in 4.2 and two disjoint trades of volume 4.


\section{ $\bmi{\st(v,s)}$ for $\bmi{v=7}$}

In this section we  treat the exceptional case of $v=7$.

\begin{pro} An $\st(7,s)$ exists if and only if $s\in\{0,2,3,5,6,8,14\}$.
\end{pro}
\begin{proof}{Let $X=\{1,\ldots,7\}$. For $s\in\{2,3,5,6,8\}$, $(X,\B_s)$ is an $\st(7,s)$, where
{\small\begin{align*}
 \B_2&=\{123, 145, 147, 156, 167, 246, 257, 346, 357, -157, -146\},\\
  \B_3&=\{123, 147, 156, 235, 246, 257, 267, 347, 367, 456, -467,-256,-237\},\\
   \B_5&=\{ 123, 136, 145, 147, 167, 235, 246, 267, 346, 347, 357, 456, -467, -345,-236,-146,-137\},\\
    \B_6&=\{123, 136, 145, 147, 156, 234, 235, 246, 257, 267, 357, 367, 456,-567,-245,-237,-236,-146,-135\},\\
     \B_8&=\{123, 145, 147, 156, 167, 234, 235, 246, 257, 267, 346, 357, 367, 456, 457,\\ &~~~~~ -567,-467,-345,-245,-237,-236,-157,-146\}.
\end{align*}}
An $\st(7,14)$ can be constructed as follows. Let $\F_1$ and $\F_2$ be the block sets of two disjoint Fano planes. Then $(X,\F_3)$  with $\F_3={X\choose3}\setminus(\F_1\cup\F_2)$ is a 2-$(7,3,3)$ design. It follows that $\F_3\cup-(\F_1\cup\F_2)$ is the block set of an $\st(7,14)$.

It remains to show that there does not exist $\sts(7,s)$ for  $s\in\{1,4,7,10\}$.

 First we show that there is no $\st(7,1)$. Let $\B^-=\{123\}$. Then two triples in $\B^+$ contain 12, say 124 and 125. Also
 two triples contain 13 which besides 1 have no points in common
with the triples containing 12, so they must be 136 and 137. Now there is
no way to choose the triples containing 23.

We proceed by proving the non-existence of $\st(7,4)$.  Let $(X,\B)$ be an $\st(7,4)$.
For all $x\in X$, we have $0\le m(x,\B^-)\le4$.
 Suppose that $m(1,\B^-)=4$. Then the number of pairs $1x$ occurring in $\B^-$ is 8, so the number of pairs  $1x$ occurring in $\B^+$ is at least 16, which implies $m(1,\B^+)\ge8$ but this is impossible since $m(1,\B^+)=3+m(1,\B^-)$. So $m(x,\B^-)\le3$ for all $x$. Suppose $m(1,\B^-)=3$. We claim that $m(1x,\B^-)\le1$ for  every $x\in X$. Let $m(12,\B^-)=2$ and $123,124\in\B^-$. Then out of six triples of $\B^+$ containing 1, four triples contain 12, and both 13 and 14 must occur twice in the two remaining  triples which is impossible. Thus, for all $x\in X$, $m(1x,\B^-)=1$, i.e. $123,145,167\in\B^-$. Now each of the pairs $23,45,67$ must appear twice in the triples of $\B^+$ not containing 1,
  but we have only five triples, hence we are done.  Therefore, $0\le m(x,\B^-)\le2$. We claim that $m(\al,\B^-)\le1$ for all $\al\in{X\choose2}$. Assume that $m(12,\B^-)=2$. Let $123,124\in\B^-$. As $m(1,\B^+)=m(1,\B^-)+3=5$, we have five triples in $\B^+$ which must contain four copies of 12, and two copies of 13 and 14; but this is also impossible.
As $m(\al,\B^-)\le1$ for all $\al\in{X\choose2}$, it is easily seen that either (a) for six points $m(x,\B^-)=2$ and for one point $m(x,\B^-)=0$, or (b) for five points $m(x,\B^-)=2$ and for two points $m(x,\B^-)=1$.
If the case (a) occurs, then $\B^-$ must be a `Pasch configuration', say $\{123,145,246,356\}$. So the triples containing 1 in $\B^+$ are either (aa) $124,125,134,135,167$, or
(ab) $124,125,135,137,146$. Now we treat the triples containing 4. If (aa) is the case, then the rest of triples containing 4 must be $456,457,246$ which is impossible because $246\in\B^-$.
If (ab) is the case, then the rest of  triples containing 4 are either (aba) $456,457,234$, or (abb) $456,345,247$. The case (aba) is impossible because it forces 5 to be in the same triple as 3 and 6.
The case (abb) forces $567\in\B^+$ and so we have  two triples for 2 which contain each of 23 and 26 twice, a contradiction.
Now suppose (b) occurs. It is easily seen that $\B^-$ must be isomorphic to $\{123,145,246,357\}$.
So the triples containing 1 in $\B^+$ are either (ba) $124,125,134,135,167$, or (bb) $125,127,135,134,146$.
The case (ba) forces $457,456,246\in\B^+$, and this implies that $357\in\B^+$ which is impossible since $357\in\B^-$. The case (bb) again is impossible because 4 must appear in three more triples in which each of 24 and 45 must occur twice but 25 must not occur.

 Let $(X,\B)$ be an $\st(7,7)$. We denote ${X\choose3}\setminus\B$ by $\overline\B$. First, suppose that  $m(\al,{\overline\B})\le2$ for all $\al\in{X\choose2}$. Since $\sum_{\al\in{X\choose2}}m(\al,{\overline\B})=3\cdot|{\overline\B}|=42$, it follows that $m(\al,{\overline\B})=2$ for all $\al\in{X\choose2}$. It turns out that $(X,{\overline\B})$ is a 2-$(7,3,2)$ design and consequently $\B^+$ and $\B^-$ are 2-$(7,3,2)$ and 2-$(7,3,1)$ designs, respectively.
 But it is well known that such a partition of ${X\choose3}$ does not exist (\cite{cr}). Therefore, for some $\al$, $m(\al,{\overline\B})=4$.
 Note that if $m(\al,{\overline\B})=0$, 2, or 4, then $m(\al,\B^+)=3,$ 2, 1, respectively. Let $\y_1$ and $\y_2$ be the characteristic vectors of $\overline\B$ and $\B^+$, respectively, and  $\x=\y_1-\y_2$. Let $\x=\y_1-\y_2$.
The entries of $W_{2,3}^7\x$ are $-3$, 0, or 3, and thus $\x\in{\rm null}_{\mathbb{Z}_3}(W_{2,3}^7)$.
If $\x_1$ is the characteristic vector of the block set of a 2-$(7,3,3)$  design, then  $W_{2,3}^7\x_1=(3,\ldots,3)^\top$, and so $\x_1\in{\rm null}_{\mathbb{Z}_3}(W_{2,3}^7)$.
It is known that the rank of $W_{2,3}^7$ over $\mathbb{Z}_3$ is one less than its rank over $\mathbb{R}$ (\cite{w2}, see also \cite{gkmm}).
It follows that for any vector $\x\in{\rm null}_{\mathbb{Z}_3}(W_{2,3}^7)$, one has $\x=c\x_1+\x_0$ for some $c\in\mathbb{R}$ and $\x_0\in{\rm null}_{\mathbb{R}}(W_{2,3}^7)$. Then, considering the inner product with the all-1 vector, we have  $0=21c+0$. This means $c=0$ which is impossible.

 Finally, we show that there is no $\st(7,10)$. We claim that if $m(\al,{\overline\B})>0$, then $m(\al,{\overline\B})=2$.
 Let $m(12,{\overline\B})=4$ and $123,124,125,126\in\overline\B$. The pairs $1x$ and $2x$, for $x=3,4,5,6$, must appear one more time in the triples of $\overline\B$ which implies that $134,156,234,256\in\overline\B$. By this, all the eight triples of $\B$ have been determined while $m(x,\B)=3$, for $x=3,4,5,6$, which is  impossible, and therefore proving the claim.
Clearly, if $m(x,\overline\B)>0$, then $m(x,{\overline\B})\ge4$.  Suppose that $m(1,\overline\B)=6$. Then $m(1x,\overline\B)=2$ for all $1\ne x\in X$. It follows that the two triples of $\B$ not containing 1, must be 234 and 567, and thus
 $m(x,\overline\B)=3$ for $1\ne x$ which is also impossible. Thus  if $m(x,{\overline\B})>0$, then $m(x,{\overline\B})=4$.
 Therefore, $\overline\B$ has the property that for exactly six points $x\in X$, $m(x,\B)=4$ and further if $m(\al,{\overline\B})>0$, then $m(\al,{\overline\B})=2$. It is easily seen that $\overline\B$ is the union of a Pasch configuration and its counterpart, say $P\cup P'$. If $m(\al,{\overline\B})=0$, then $m(\al,\B^-)=2$ and $m(\al,\B^+)=3$; and if $m(\al,{\overline\B})=2$, then $m(\al,P)=m(\al,P')=m(\al,\B^-)=1$ and $m(\al,\B^+)=2$. Hence  $\B^-\cup P$ and $\B^+\cup P'$ are 2-$(7,3,2)$ and 2-$(7,3,3)$  designs, respectively. It is known that the block set of any 2-$(7,3,2)$ designs is a union of two disjoint Fano planes, say $F\cup F'$. On the other hand, any Pasch configuration can be extended to a Fano plane in a unique way by adding a set of three triples. Let $Q$ be such a set for $P$. We claim that $Q\subset F\cup F'$. To obtain a contradiction, assume that $Q\cap \B^+\ne\emptyset$. (Note that $Q\cap(P\cup P')=\emptyset$.) If $|Q\cap\B^+|=3$, then $P'\cup Q$ is a Fano plane contained in $B^+\cup P'$, which is impossible because it is well known that any 2-$(7,3,3)$ design does not contain a Fano plane. If $|Q\cap\B^+|=2$, then $P\cup Q$ has five points in common with $F\cup F'$, but this is impossible since any two distinct Fano planes have 0, 1, or 3 blocks in common. If $|Q\cap\B^+|=1$, then
 $B^+\cup P\cup P'\cup Q$ is a 2-$(7,3,4)$ design with 27 distinct blocks in which one block is  repeated; but such a design does not exist (\cite{hl}).

Now the proof is  complete.
}\end{proof}


\begin{table}
\textbf{\large Appendix.}~$\T(2,3,8)$ trades as described in Lemma~\ref{remain}. \\
{\small The rows are indexed by triples and the columns are indexed by the volume of trades.}

$${\footnotesize
  \begin{array}{ccccccccccccccccc}
  \hline
    &4 & 6&7 &8 &9 &10 &11 &12 & 13& 14& 15& 16&17 &18 &20 &24 \\
    \hline
    123& & & & & & & & & & & & & & & & \\ 124& & & & & & & & & & & & & & & & \\ 134& & & & & & & & & & & & & & & & \\ 234& & & & & & & & & & & & & & & & \\ 125& & & & +& & +& & +& +& & +& +& +& +& +& +\\ 126& & & & -& & -& & -& -& & -& -& -& -& -& -\\ 127& & & & & & & -& -& -& -& & -& -& -& -& -\\ 128& & & & & & & +& +& +& +& & +& +& +& +& +\\ 135& & & & -& & & & -& -& & -& -& -& -& -& -\\ 136& & & & +& -& +& & +& & & +& +& +& +& +& +\\ 137& & & & & +& & & -& & -& -& & -& -& -& -\\ 138& & & & & & -& & +& +& +& +& & +& +& +& +\\ 145& & & & & & & & & & & & +& +& +& +& +\\ 146& & & & & +& & & & +& & & & & & +& +\\ 147& & & & & -& & +& +& & +& +& & & & -& -\\ 148& & & & & & & -& -& -& -& -& -& -& -& -& -\\ 156& & +& +& & +& & +& +& +& +& +& & & & -& -\\ 157& & -& -& & -& -& -& & & & -& & & & +& +\\ 158& & & & & & & & -& -& -& & -& -& -& -& -\\ 167& & & & & & & & & & & & & +& +& +& +\\ 168& & -& -& & -& & -& -& -& -& -& & -& -& -& -\\ 178& & +& +& & +& +& +& +& +& +& +& +& +& +& +& +\\ 235& & -& & & & & & & & +& +& +& & +& +& +\\ 236& & & & & & & & & & -& -& -& & -& -& -\\ 237& & & & & & & +& +& +& +& & +& +& +& +& +\\ 238& & +& & & & & -& -& -& -& & -& -& -& -& -\\ 245& & & & & & & & & & -& -& -& -& -& -& -\\ 246& -& & -& -& -& -& -& & -& & & & & & -& -\\ 247& & & & & & & & & & & & & & +& +& +\\ 248& +& & +& +& +& +& +& & +& +& +& +& +& & +& +\\ 256& & & & & & & & & & & & & & & +& +\\ 257& & +& & & & & & & & & & & +& -& -& -\\ 258& & & & -& & -& & -& -& & -& -& -& & -& -\\ 267& +& & +& +& +& +& +& & +& +& +& +& & +& +& +\\ 268& & & & +& & +& & +& +& & +& +& +& +& +& +\\ 278& -& -& -& -& -& -& -& & -& -& -& -& -& -& -& -\\ 345& & +& & & & -& & & & & & -& & -& -& -\\ 346& +& & +& +& +& +& +& & +& +& +& +& & +& +& +\\ 347& & & & & & & -& & & & & & +& & +& +\\ 348& -& -& -& -& -& & & & -& -& -& & -& & -& -\\ 356& & & -& & & & & & & & & +& +& +& & +\\ 357& & & +& & & & & & & -& & -& -& & & -\\ 358& & & & +& & +& & +& +& & & +& +& & +& +\\ 367& -& & -& -& -& -& -& & -& & & -& -& -& -& -\\ 368& & & +& -& +& -& & -& & & -& -& -& -& & -\\ 378& +& & & +& & +& +& & & +& +& +& +& +& & +\\ 456& & -& & & -& & -& -& -& -& -& -& -& -& & -\\ 457& & & & & +& +& +& & & +& +& +& & +& & +\\ 458& & & & & & & & +& +& +& +& +& +& +& +& +\\ 467& & & & & & & & & & -& -& & & -& -& -\\ 468& & +& & & & & +& +& & +& +& & +& +& & +\\ 478& & & & & & -& -& -& & -& -& -& -& -& & -\\ 567& & & & & & & & & & & & & & & & \\ 568& & & & & & & & & & & & & & & & \\ 578& & & & & & & & & & & & & & & & \\ 678& & & & & & & & & & & & & & & &\\ \hline
  \end{array}}$$
  \end{table}
\end{document}